\documentclass[11pt]{article}

\usepackage[margin=1in]{geometry}
\usepackage{amsmath,amssymb,bm}
\usepackage{mathrsfs,mathabx}
\usepackage{amsthm}
\usepackage{hyperref}
\usepackage{graphicx}
\usepackage{enumerate} 
\usepackage{enumitem}
\numberwithin{equation}{section} 

\usepackage[comma,sort, numbers]{natbib}

\usepackage{algorithm2e}
\usepackage[utf8]{inputenc}
\usepackage{graphicx,tikz,tcolorbox}
\usepackage[]{algorithm2e}

\newtheorem{lemma}{Lemma}
\newtheorem{theorem}{Theorem}
 
\newtheorem{remark}{Remark}

\allowdisplaybreaks
\title{On the Number of Almost Empty Monochromatic Triangles}

\author{Bhaswar B. Bhattacharya\thanks{Department of Statistics and Data Science, University of Pennsylvania and National University of Singapore,  \texttt{bhaswar@wharton.upenn.edu}} \and 
Sandip Das\thanks{Advanced Computing and Microelectronics Unit, Indian Statistical Institute, Kolkata, \texttt{sandip.das.69@gmail.com}}
\and 
Sk Samim Islam\thanks{Advanced Computing and Microelectronics Unit, Indian Statistical Institute, Kolkata, \texttt{samimislam08@gmail.com}}
\and 
Aashirwad Mohapatra\thanks{Advanced Computing and Microelectronics Unit, Indian Statistical Institute, Kolkata, \texttt{aashirwad\_r@isical.ac.in}} 
\and 
Ishan Paul\thanks{Department of Statistics, University of Michigan, \texttt{ishanpl@umich.edu}} 
\and 
Saumya Sen\thanks{Advanced Computing and Microelectronics Unit, Indian Statistical Institute, Kolkata, \texttt{saumyasen72@gmail.com}}
}
\date{}

\begin{document}

\maketitle

\begin{abstract} 
In this paper, we consider the problem of counting almost empty monochromatic triangles in colored planar point sets, that is, triangles whose vertices are all assigned the same color and that contain only a few interior points. Specifically, we show that any $c$-coloring of a set of $n$ points in the plane, no three on a line, contains $\Omega(n^2)$ monochromatic triangles with at most $c-1$ interior points and $\Omega(n^{\frac{4}{3}})$ monochromatic triangles with at most $c-2$ interior points, for any fixed $c \geq 2$. The latter, in particular, generalizes the result of \citet{pach2013monochromatic} on the number of monochromatic empty triangles in 2-colored point sets, to the setting of multiple colors and monochromatic triangles with a few interior points. We also derive the limiting value of the expected number of triangles with $s$ interior points in uniform random point sets, for any integer $s \geq 0$. As a result, we obtain the expected number of monochromatic triangles with at most $s$ interior points in random colorings of random point sets. 
\end{abstract}

\section{Introduction }

The  Erd\H{o}s-Szekeres theorem \cite{erdos1935combinatorial} is a celebrated result in combinatorial geometry and Ramsey theory, which asserts that for every positive integer $r \geq 3$, there exists a smallest integer $\mathsf{ES}(r)$, such that any set of at least $\mathsf{ES}(r)$ points in the plane in general position, that is, no three on a line, contains $r$ points in convex position. In particular, Erd\H{o}s and Szekeres \cite{erdos1935combinatorial,erdHos1960some} established the following bounds:
$$2^{r - 2}+1 \leq \mathsf{ES}(r) \leq {{2r - 4} \choose {r - 2}} + 1.$$
They also famously conjectured that the lower bound is an equality. 
%They also famously conjectured that the lower bound is sharp. 
After a series of improvements, in a major breakthrough \citet{suk2017polygon} 
almost resolved this conjecture by showing that
$$\mathsf{ES}(r) = 2^{r+ O(r^{\frac{2}{3}} \log r ) }.$$ 
Later, \citet{holmsen2020two} improved the exponent to $r+O(\sqrt{r \log r})$.

A stronger version of the above problem, formulated by Erd\H{o}s in \cite{erdos1978some}, asks whether every sufficiently large planar point set in general position contains an $r$-{\it hole}, that is, a convex $r$-gon whose interior contains no other points of the set. To this end, denote by $\mathsf{H}(r)$ the minimum number (if it exists) such that any set of $\mathsf{H}(r)$ points in the plane in general position contains an $r$-hole, for $r \geq 3$. Trivially, $\mathsf{H}(3)=3$; Esther Klein initially showed that $\mathsf{H}(4)=5$; and \citet{harborth1978konvexe} proved that $\mathsf{H}(5)=10$. \citet{horton1983sets} showed that $\mathsf{H}(r)$ does not exist for $r \ge7$ by constructing arbitrarily large point sets without a 7-hole. The finiteness of $\mathsf{H}(6)$, after being open for several years,  was established independently by \citet{gerken2008empty} and \citet{nicolas2007empty}. In a recent breakthrough, \citet{heule2024happy} proved that $\mathsf{H}(6) = 30$, finally resolving the 6-hole problem.

The non-existence of $\mathsf{H}(r)$, for $r \geq 7$, prompted the study of various relaxations. One such relaxation is the notion of almost empty polygons, that is, convex polygons containing few interior points \cite{nyklova2003almost,koshelev2011computer,huemer2022weighted} (which is also closely related to the notion of islands \cite{bautista2011computing,balko2022holes}).
A weaker restriction where the number of interior points is divisible by a certain number has been studied in \cite{bialostocki1991some}. 
%More formally, for integers $k \geq 3$ and $s \geq 0$, let $H(k,s)$ denote the smallest integer (if it exists) such that every set of at least $H(k,s)$ points in the plane, in general position, contains a convex $k$-gon with at most $s$ interior points. Obviously, $\mathsf{ES}(k) \leq H(k,s) \leq H(k,0) = \mathsf{H}(r)$. Using a Horton set construction, \citet{nyklova2003almost} showed that 
%obtained the values of proved that $\mathsf{H}(6,6) = \mathsf{ES}(6) = 17$ and  $\mathsf{H}(6,5) = 19$; she also proved that 
%$H(k,s)$ does not exist when $s \leq 2^{(k+6)/4} - k - 3$. 
In another direction, \citet{devillers2003chromatic} studied colored variations of the Erd\H{o}s-Szekeres problem. In these problems, a set of points is colored with $c \ge 2$ colors, and a polygon spanned by points in the set is called {\it monochromatic} if all its vertices have the same color. In this context, \citet{grima2009some} showed that any set of 9 points arbitrarily colored with 2 colors contains a monochromatic 3-hole. On the other hand, \citet{devillers2003chromatic} demonstrated the existence of arbitrarily large 3-colored point sets with no monochromatic 3-hole, as well as arbitrarily large 2-colored point sets with no monochromatic 5-hole. They also conjectured that every sufficiently large 2-colored point set contains a 4-hole. Although this conjecture remains open, \citet{aichholzer2010large} established a relaxed version in which the 4-hole is not required to be convex. Other chromatic variantions, such as balanced and rainbow holes, have also been studied (see, for example, \cite{arevalo2022rainbow,diaz2021note,fabila2021empty,kano2021discrete} and the references therein).

%While this conjecture remains open, a relaxation that does not require the 4-hole to be convex has been proved by \citet{aichholzer2010large}.

Combining the notions of almost empty polygons and monochromatic holes,  \citet{colorempty} initiated the study of almost empty monochromatic triangles. Formally, given integers $r \geq 3$, $c \geq 2$, and $s \geq 0$, define $\mathsf{M}_r(c,s)$ to be the least integer (if it exists) such that any set of at least $\mathsf{M}_r(c, s)$  points in the plane in general position colored with $c$ colors contains a monochromatic convex $r$-gon with at most $s$ interior points. Moreover, denote by $\lambda_r(c)$ the least integer such that $\mathsf{M}_r(c, \lambda_r(c)) < \infty$. \citet{colorempty} considered the case of triangles, that is, $r=3$, and proved that, for any $c \geq 2$, 
\begin{align}\label{eq:trianglecolorempty}
\left\lfloor \frac{c-1}{2} \right\rfloor \leq \lambda_3(c) \leq c-2. 
\end{align}
In particular, for $c=3$, this means any sufficiently large 3-colored point set contains a monochromatic triangle with at most 1 interior point. For $c \geq 4$, Cravioto et al. \cite{cravioto2019almost} improved the upper bound in \eqref{eq:trianglecolorempty} to $c-3$. They also showed that $\lambda_4(c) \leq 2c-3$, for $c \geq 2$ (see \cite[Theorem 3]{cravioto2019almost}). The upper bound on $\lambda_4(c)$ can be improved to $2c-4$ if the convexity requirement is relaxed 
\cite{bhattacharya2025almost}.

Existence problems for convex polygons (including their empty, almost-empty, and colored variants) are closely linked to the question of counting such configurations. This dates back to \citet{erdos1973crossing}, where the problem of determining the number of convex $r$-gons contained in any set of $n$ points was first posed. Trivially, for $r=3$, the number of triangles in any set of $n$ points in the plane in general position is ${n \choose 3}$. If, in addition, the triangles are required to be empty, then
\citet{katchalski1988empty} showed that any set of $n$ points in the plane contains at least ${{n-1} \choose 2}$ empty triangles and also constructed sets with at most $O(n^2)$ empty triangles. Thereafter, the problem of counting the number of empty triangles and its variations has been extensively studied (see \cite{barany2004planar,barany1987empty,barany2013many,erdos1992geometry,garcia2011note,pinchasi2006empty,researchbook,dehnhardt1987leere,valtr1992,aichholzer2020superlinear,triangles,balko2023tight,dumitrescu2000planar,aichholzer2014lower} and the references therein). In particular,  suppose  
$Z(A)$ denotes the number of empty triangles in a set $A$ of $n$ points and $Z(n)= \min_{|A| = n} Z(A)$, where the minimum is taken over all sets $A \subset \mathbb R^2$, with $|A| =n$, in general position. Then the current best bounds are: 
$$n^2 + \Omega(n \log^{\frac{2}{3}} n) \leq Z(n) \leq 1.6196 n^2 + o(n^2) , $$
where the upper bound is from \citet{barany2004planar} and the lower bound is due to \citet{aichholzer2020superlinear}. There has also been a series of papers studying empty triangles and, more generally, empty simplices, in random point sets \cite{reitzner2024stars,valtr1992,balko2022holes}. Specifically, if $S$ is a random collection of $n$ points chosen uniformly and independently from a planar convex set of area~1, then $\mathbb E[Z(S)] = 2n^2 + o(n^2)$ (see \cite[Theorem 1.4]{reitzner2024stars}). 
In the chromatic setting, complementing the existential results in \cite{garcia2011note,devillers2003chromatic}, \citet{aichholzer2009empty} showed that any 2-colored planar point set in general position of size $n$, contains $\Omega(n^{\frac{5}{4}})$ empty monochromatic triangles, for $n$ sufficiently large. This bound was improved to $\Omega(n^{\frac{4}{3}})$ by \citet{pach2013monochromatic}. On the other hand, the magnitude of the upper bound remains at $O(n^2)$ which, in particular, can be obtained by randomly 2-coloring a random collection of $n$ points in the unit square. Lower bounds on the number of empty monochromatic simplices are derived in \citet{aichholzer2014empty}.

In this paper, we initiate the study of counting almost empty monochromatic triangles in planar point sets. Specifically, we obtain two lines of results: (1) lower bounds on the number of almost empty monochromatic triangles (Section \ref{sec:lowerbound}), and (2) precise asymptotics for the expected number of almost empty triangles in random point sets (Section \ref{sec:random}). Using the latter, we derive the asymptotics of the expected number of almost empty monochromatic triangles in  random colorings of random point sets. 
%This, in particular, implies an $O(n^2)$ upper bound on the number of monochromatic triangles with at most $s$ interior points in an $n$-point set, for any $s \ge 0$.

\subsection{Lower Bounds for the Number of Almost Empty Monochromatic Triangles }
\label{sec:lowerbound}

Given a finite $c$-colored point set $P\subset\mathbb{R}^2$ in
general position and an integer $s\geq 0$, denote by $X_{\leq s}^{(c)}(P)$ the number of monochromatic triangles determined by $P$ that contain at most $s$ points of $P$ in their interior. Define
$$X_{\leq s}^{(c)}(n) := \min_{\substack{P\subset\mathbb{R}^2,\ |P|=n\\ P\text{ a $c$-colored point set}}} X_{\leq s}^{(c)}(P),$$
where the minimum is taken over all $c$-colored point sets $P \subset \mathbb R^2$, with $|P|=n$, in general position. Note that $X_{\leq 0}^{(c)}(n)$ corresponds to the minimum number of monochromatic empty triangles in any $c$-coloring of an $n$-point set in general position. In this notation, the results of \citet{aichholzer2009empty} and \citet{pach2013monochromatic} can be stated as: $X_{\leq 0}^{(2)}(n) =\Omega(n^{\frac{5}{4}})$  and $X_{\leq 0}^{(2)}(n) =\Omega(n^{\frac{4}{3}})$, respectively. Our first result is a lower bound on $X_{\leq (c-1)}^{(c)}(n)$.

\begin{theorem}\label{thm:c_c-11}   
For $c \geq 2$ fixed, $X_{\leq (c-1)}^{(c)}(n) \geq L n^2$, for some constant $L= L(c)$ depending on $c$ and $n$ sufficiently large. 
\end{theorem}

Theorem~\ref{thm:c_c-11} can be proved by a simple argument that counts the number of monochromatic triangles that can contain $c$ or more points in a star triangulation centered at a point (see Section~\ref{sec:c_c-11pf} for details). Later, in Section \ref{sec:random} we show that the expected number of triangles with at most $s$ interior points in a sample of $n$ independent uniform points is $O(n^2)$. In particular, this implies the existence of point sets with $O(n^2)$ such triangles, hence, the dependence on $n$ in Theorem~\ref{thm:c_c-11} is best possible.

\begin{remark} 
{\em Specifically, for $c=2$, Theorem~\ref{thm:c_c-11} shows that any $2$-colored point set contains $\Omega(n^2)$ monochromatic triangles with at most one interior point. In comparison, recall that for monochromatic empty triangles in 2-colored point sets the best lower bound is $\Omega(n^{\frac{4}{3}})$ (due to \citet{pach2013monochromatic}). While improving this lower bound, possibly up to $\Omega(n^2)$, remains open, Theorem~\ref{thm:c_c-11} shows that if the emptiness condition is relaxed to allow at most one interior point, an asymptotically tight $\Omega(n^2)$ lower bound can be obtained easily.  } 
\end{remark}

Next, we consider monochromatic triangles with at most $c-2$ interior points in $c$-colored point sets. The upper bound in \eqref{eq:trianglecolorempty} readily implies a linear lower bound on $X_{\leq (c-2)}^{(c)}(n)$. In the next theorem, we establish a superlinear lower bound.

\begin{theorem}\label{thm:k_general}   
For $c \geq 2$ fixed, $X_{\leq (c-2)}^{(c)}(n) \geq L n^{\frac{4}{3}}$, for some constant $L= L(c)$ depending on $c$ and $n$ sufficiently large. 
\end{theorem}

The proof of Theorem \ref{thm:k_general} is given in Section \ref{sec:k_generalpf}. Note that Theorem \ref{thm:k_general} for $c=2$ (which corresponds to the number of monochromatic empty triangles in 2-colored point sets) is the result of \citet{pach2013monochromatic}. Theorem \ref{thm:k_general} generalizes this to monochromatic triangles in $c$-colored point sets with at most $c-2$ interior points, for any $c \ge 2$. A key observation in the proof is that the difference between the maximum and the average cardinalities of the color classes is the natural notion of discrepancy when counting monochromatic triangles with $c-2$ interior points. The argument then proceeds by recursively removing points from the set in a way such that one of the following situations occurs: (1) one color class contains significantly more points than the average (that is, the point set has large discrepancy), which yields the desired number of monochromatic triangles with at most $c-2$ interior points; or (2) there exists a linear number of points, each of which is incident on many monochromatic triangles with at most $c-2$ interior points. Optimizing the thresholds, one obtains $\Omega(n^{\frac{4}{3}})$ monochromatic triangles with at most $c-2$ interior points in either case.

\begin{remark}
{\em An important special case of Theorem \ref{thm:k_general} is when $c=3$. Recall that \cite[Theorem~3.3]{devillers2003chromatic} showed that there exist arbitrarily large 3-colored point sets with no monochromatic empty triangles. On the other hand, monochromatic triangles with at most 1 interior point exist in any 3-coloring of a sufficiently large point set. Specifically, \citet{colorempty}  showed that $\mathsf{M}_3(3, 1) = 13$, which readily implies a linear lower bound on $X_{\leq 1}^{(3)}(n)$. Theorem \ref{thm:k_general} for $c=3$, implies the super-linear lower bound $X_{\leq 1}^{(3)}(n)= \Omega(n^{\frac{4}{3}})$, that is, any 3-coloring of a set of $n$ points contains at least $\Omega(n^{\frac{4}{3}})$ monochromatic triangles with at most 1 interior point. }
\end{remark}

\subsection{Almost Empty Triangles in Random Point Sets }
\label{sec:random}

For any integer $s \geq 0$ and a finite set $A \subset \mathbb R^2$, denote by $Z_{=s}(A)$ the number of triangles with vertices in $A$ that contain exactly $s$ points of $A$ in its interior. Clearly, 
\begin{align}\label{eq:XA}
Z_{\leq s} (A) = \sum_{r=0}^s Z_{=r} (A). 
\end{align}
(Note that $Z_{\leq 0}(A) = Z_{= 0}(A)$ corresponds to the number of empty triangles in $A$, which is the same as $Z(A)$ defined in the Introduction.) In the following result we derive the limiting value of the expectation of $Z_{=s}(\cdot)$, and consequently $Z_{\leq s}(\cdot)$, for a collection of randomly chosen points. 

\begin{theorem}\label{thm:U}
Suppose $\{U_1, U_2, \ldots, U_n\}$ be a collection of independent random points
uniformly distributed on $[0, 1]^2$. Then, for any fixed integer $s \geq 0$, 
\begin{align}\label{eq:XA2}
\lim_{n \rightarrow \infty} \frac{\mathbb E[Z_{=s}(\{U_1, U_2, \ldots, U_n\})]}{n^2}  = 2 . \end{align}
Consequently, from \eqref{eq:XA},  
\begin{align}\label{eq:XA2s}
\lim_{n \rightarrow \infty} \frac{\mathbb E[Z_{ \leq s}(\{U_1, U_2, \ldots, U_n\})]}{n^2}  = 2 (s+1) . 
\end{align} 
\end{theorem}

The proof of Theorem \ref{thm:U} is given in Section \ref{sec:Upf}. Note that the result in \eqref{eq:XA2} for $s=0$ (which corresponds to the number of empty triangles) is known from \cite[Theorem 1.4]{reitzner2024stars}. Theorem \ref{thm:U} generalizes this to triangles with a fixed number of interior points.

\begin{remark} 
{\em A natural way to randomly color a set $P = \{p_1, p_2, \ldots, p_n\}$ of $n$ points in 
$\mathbb R^2$ is to independently assign to each point $p_i \in P$ a color $a \in \{1, 2, \ldots, c\}$ uniformly at random. More formally, 
\begin{align}\label{eq:color}
\mathbb P( p_i \in P \text{ is assigned color } a \in  \{1, 2, \ldots, c\}) = \frac{1}{c}, 
\end{align}
independently for each $1 \leq i \leq n$. Denote this random coloring by the map $\phi: P \rightarrow   \{1, 2, \ldots, c\}$, where $\phi(p)$ is the color of the point $p \in P$. Then the number of monochromatic triangles with at most $s$ interior points can be expressed as: 
\begin{align*} 
X^{(c)}_{\leq s}(P, \phi) := \sum_{1 \leq i < j < k \leq n} \bm 1
\{ |\mathrm{int}(\Delta(p_i, p_j, p_k)) \cap \{p_1, p_2, \ldots, p_n\}| \leq s \} \bm 1\{\phi(p_i) = \phi(p_j) = \phi(p_k) \}, 
\end{align*}
where $\mathrm{int}(\Delta(p_i, p_j, p_k))$ is the interior of the triangle with vertices $p_i, p_j, p_k$. Note that 
$$\mathbb P(\phi(p_i) = \phi(p_j) = \phi(p_k) ) = \frac{1}{c^2},$$
for all $1 \leq i < j < k \leq n$. Here, the probability is taken over the randomness of the coloring, given the point set $P$. Hence, taking expectation over the randomness of the coloring $\phi$ gives, 
\begin{align*} 
\mathbb E [X^{(c)}_{\leq s}(P, \phi) ] = \frac{1}{c^2}\sum_{1 \leq i < j < k \leq n} \bm 1
\{ |\mathrm{int}(\Delta(p_i, p_j, p_k)) \cap \{p_1, p_2, \ldots, p_n\}| \leq s \} 
& = \frac{1}{c^2} Z_{\leq s}(P) . 
\end{align*}
Hence, if we consider a random collection of $n$ independent uniform points $\{U_1, U_2, \ldots, U_n\}$ in $[0, 1]^2$ and color them uniformly at random as in \eqref{eq:color} (given the realization of the point set), then taking joint expectation over the randomness of the points and the coloring gives, 
\begin{align}\label{eq:pointsrandomcolor}
\mathbb E [X^{(c)}_{\leq s}(\{U_1, U_2, \ldots, U_n\}, \phi) ]  = \frac{1}{c^2} \mathbb E [Z_{\leq s}(\{U_1, U_2, \ldots, U_n\}) ] = \frac{2(s+1)}{c^2} n^2 (1+ o(1)) , 
\end{align}
where the second equality follows from Theorem \ref{thm:U}. This shows that a random $c$-coloring of a random collection of $n$ points contains, on expectation, $O(n^2)$ monochromatic triangles with at most $s$ interior points, for $c \ge 2$ and $s \ge 0$ fixed, with the leading constant of the expected value given by \eqref{eq:pointsrandomcolor}. This implies the existence of $c$-colored $n$-point sets with $O(n^2)$ monochromatic triangles with at most $s$ interior points, for $s \ge 0$. } 
\end{remark}

\section{ Proofs of the Results }

In this section we prove the results stated in the previous section. We begin by collecting some notations and definitions in Section \ref{sec:defintions}. The proofs of Theorems \ref{thm:c_c-11}, \ref{thm:k_general}, \ref{thm:U} are then presented in Sections \ref{sec:c_c-11pf}, \ref{sec:k_generalpf}, \ref{sec:Upf}, respectively.

\subsection{Notations}
\label{sec:defintions}

To begin with we collect some notations and definitions that will be used throughout the paper. For any set $A\subseteq\mathbb{R}^2$, denote by $\lambda_2(A)$ the area of $A$. Similarly, for any line segment $L\subseteq\mathbb{R}$ or $L\subseteq\mathbb{R}^2$, $\lambda_1(L)$ will denote the length of $L$. Also, $\mathrm{int}(A)$ will denote the interior of the set $A$. 
 
Given a finite set $S$ of points in the plane, $\Delta(s_i s_j s_k)$ will denote the triangle with vertices $s_i, s_j, s_k \in S$. The triangle $\Delta(s_i s_j s_k)$ is said to be $r$-{\it blocked}, for some integer $r \geq 0$, if $|\mathrm{int}(\Delta(s_i s_j s_k)) \cap S| \ \geq r$. Furthermore, $CH(S)$ will denote the convex hull of the set $S$, and $V(CH(S))$ will denote the set of vertices (the extreme points) of $CH(S)$.

For two sequences $\{a_n\}_{n \geq 1}$ and $\{b_n\}_{n \geq 1}$, we write $a_n \gtrsim b_n$ whenever $a_n \geq L b_n$, for all $n$ large enough, where $L = L(c) > 0$ is a constant that may depend on the number of colors $c$. Also, $a_n \sim b_n$ means $a_n /b_n \rightarrow 1$, as $n \rightarrow \infty$.

\subsection{Proof of Theorem \ref{thm:c_c-11} }
\label{sec:c_c-11pf}

Let $P = \{p_1, p_2, \ldots, p_n\}$ be a set of $n$ points in the plane in general position. Suppose each point in $P$ is assigned a color from the set $[c]:= \{1, 2, \ldots, c\}$. Denote by $\phi(p_i)$ the color of the point $p_i \in P$ and let $P_a$ be the subset of points of color $a \in [c]$, that is, 
\begin{align}\label{eq:colorP}
P_a := \{p \in P: \phi(p) = a\} .  
\end{align}
Clearly, $P_1, P_2, \ldots, P_c$ are disjoint sets and $\sum_{a=1}^c |P_a| = n$. Without loss of generality, assume $|P_1| \geq |P_2| \geq \cdots \geq |P_c|$.

Fix a point $p \in P_1$ and consider the star connecting $p$ to all other points in $P_1 \backslash \{p\}$. Complete this star to a triangulation of $P_1$. This triangulation has at least $|P_1|-2$ triangles of color $1$ with $p$ as a vertex. Observe that among these $|P_1|-2$ monochromatic triangles, at most $\frac{n-|P_1|}{c}$ of them are $c$-blocked. Hence, for $n \geq 4 c^2$, we have at least 
\begin{align*} 
|P_1|- \frac{n-|P_1|}{c} - 2 & = \left(1 + \frac{1}{c}\right) |P_1| - \frac{n}{c} - 2 \\ 
& \geq   \left(1 + \frac{1}{c}\right) |P_1| - |P_1| - 2 \tag*{(since $|P_1| \geq \frac{n}{c}$)} \\ 
& = \frac{|P_1|}{c} - 2 \geq   \frac{|P_1|}{2c} , 
\end{align*} 
monochromatic triangles of color $1$ with $p$ as a vertex with at most $c-1$ interior points. Adding this over $p \in P_1$ and observing that each triangle is overcounted most 3 times, gives 
 $$X^{(c)}_{\leq (c-1)}(P) \geq  \frac{|P_1|^2}{6c} \gtrsim n^2 .$$
This completes the proof of Theorem \ref{thm:c_c-11}.  \hfill $\Box$

\subsection{Proof of Theorem \ref{thm:k_general} }
\label{sec:k_generalpf}

A key ingredient in the proof of Theorem \ref{thm:k_general} is the following lemma from \cite{aichholzer2009empty}.

\begin{lemma}[Order Lemma \cite{aichholzer2009empty}] \label{lemma2}
Let $\Delta(p_1p_2p_3)$ be a triangle with vertices $p_1, p_2, p_3$ containing the points $q_1, q_2, \ldots, q_m$ in its interior. Then the set $\{p_1, p_2, p_3, q_1, q_2, \ldots, q_m \}$ can be triangulated so that at least $m + \sqrt m + 1$ triangles have $p_1, p_2, p_3$ as one of their vertices.
\end{lemma}

Let $S$ be a finite set of points in the plane in general position, where each point $s \in S$ is assigned a color from the set $[c] := \{1, 2, \ldots, c\}$. Denote by $\phi(s)$ the color of the point $s$ and, for each $a \in [c]$, let $S_a \subseteq S$ be the subset of points with color $a$. The next lemma shows that if the cardinality of one of the color classes differs significantly from the average, then $S$ contains many monochromatic triangles with at most $c-2$ interior points. To this end, define the {\it discrepancy} of any set $S$ as: 
\begin{align}\label{eq:deltaS}
\delta(S) = \max_{a \in [c]} |S_a| - \frac{1}{c}\sum_{a=1}^c |S_a|  = \max_{a \in [c]} |S_a| - \frac{|S|}{c} . 
\end{align}
The following result is the version of the discrepancy lemma for monochromatic triangles with at most $c-2$ interior points, for $c \geq 2$, paralleling the analogous results for monochromatic empty triangles in \cite{aichholzer2009empty} and monochromatic empty simplices in \cite{aichholzer2014empty}.

\begin{lemma}[Discrepancy Lemma] \label{lemma3}
Let $S$ be a set of points in general position in the plane colored with $c \geq 2$ colors. If $\delta(S) > \frac{4 (c-1)}{c}$, then 
\begin{align}\label{eq:colorS}
X^{(c)}_{\leq (c-2)}(S) \geq \frac{ (\gamma(S)-2) |S| }{3c} \geq \frac{\delta(S) |S|}{ 6 c } ,  
\end{align}
where $\gamma(S) = \frac{c}{c-1} \delta(S)$.  
 \end{lemma}

 \begin{proof} Without loss of generality we can assume $|S_1| \geq |S_2| \geq \cdots \geq |S_c|$. Then the discrepancy can be written as: 
$$\delta(S) =  |S_1| - \frac{1}{c}\sum_{a=1}^c |S_a| = \left(1-\frac{1}{c}\right) |S_1| - \frac{1}{c}\sum_{a=2}^c |S_a| . $$ 
Hence, 
\begin{align*}%\label{eq:Spf}
|S_1|  =  \frac{1}{c-1} \sum_{a=2}^c |S_a| + \frac{c}{c-1} \delta(S) = \frac{1}{c-1} \sum_{a=2}^c |S_a| + \gamma(S) , 
\end{align*}
where $\gamma$ is as defined in the statement of Lemma \ref{lemma3}. Now, fix a point $s \in S_1$ and consider the star connecting $s$ to all other points in $S_1 \backslash \{s\}$. Complete this star to a triangulation of $S_1$. This triangulation has at least $|S_1|-2$ triangles of color $1$ with $s$ as a vertex. Observe that among these $|S_1|-2$ monochromatic triangles, at most $\frac{|S| - |S_1|}{c-1}$ them are $(c-1)$-blocked. Hence, we have at least 
$$|S_1|- \frac{|S| - |S_1|}{c-1} - 2  = |S_1| - \frac{1}{c-1} \sum_{a=2}^c |S_a|  -2 = \gamma(S) - 2 , $$
monochromatic triangles of color $1$ with $s$ as a vertex with at most $c-2$ interior points. Adding this over $s \in S_1$ and observing that each triangle is overcounted most 3 times, gives 
 \begin{align*}%\label{eq:XcS}
X^{(c)}_{\leq (c-2)}(S) \geq  \frac{(\gamma(S) - 2) |S_1|}{3} \geq  \frac{(\gamma(S) - 2) |S|}{3 c} , 
 \end{align*} 
 since $|S_1| \geq \frac{|S|}{c}$ (because it is the color class with the largest cardinality). 
%From \eqref{eq:Spf} and using $\frac{1}{c-1} \sum_{a=2}^c |S_a| = \frac{|S| - |S_1|}{c-1} $ gives, $  |S_1|  \geq   \frac{|S|}{c} + \frac{c-1}{c} \gamma =  \frac{|S|}{c} + \alpha$. 
%Using this bound in \eqref{eq:XcS}, the result in \eqref{eq:colorS} follows.   
This proves the first inequality in \eqref{eq:colorS}. For the second inequality in \eqref{eq:colorS} observe that $\gamma(S) - 2 \geq \frac{\gamma(S)}{2}$, for $\gamma(S) > 4$, and $\gamma(S) = \frac{c}{c-1} \delta(S) \geq \delta(S)$. Hence, $\frac{(\gamma(S) - 2) |S|}{3 c} \geq \frac{\gamma(S) |S|}{ 6 c} \geq \frac{\delta(S) |S|}{ 6 c}$. 
\end{proof}

In the proof of Theorem \ref{thm:k_general}, we will apply the above lemma for $\delta(S)$ which is of smaller order than $|S|$. Then, a consequence of the discrepancy lemma is that the cardinalities of all the color classes are close to the average (up to smaller order terms):

\begin{lemma} \label{lemma1}
Let $S$ be a set of points in general position in the plane colored with $c \geq 2$ colors. Then, for $a \in [c]$, 
\begin{align}\label{eq:Sc}
\frac{|S|}{c} - (c-1) \delta(S)  \leq |S_a| \leq \frac{|S|}{c} + \delta(S) ,\end{align}
where $\delta(S)$ is defined in \eqref{eq:deltaS}. 
\end{lemma}

\begin{proof} From \eqref{eq:deltaS} we have, for each $a \in [c]$, 
\begin{align}\label{eq:Scpf}
|S_a| \leq \max_{a \in [c]} |S_a| = \frac{|S|}{c} + \delta(S) , 
\end{align} 
which proves the upper bound in \eqref{eq:Sc}.

For the lower bound, without loss of generality assume $a=1$ and suppose, if possible,  
\begin{align*}
|S_1| < \frac{|S|}{c} - (c-1) \delta(S). 
\end{align*} 
Then using the above and the bound in \eqref{eq:Scpf}, for $a \geq 2$, implies that 
$$|S| = \sum_{a=1}^c |S_a| < \frac{|S|}{c} - (c-1) \delta(S) + \sum_{a=2}^c |S_a| \leq |S|, $$ 
which is a contradiction. Hence, the lower bound in \eqref{eq:Sc} holds, for all $a \in [c]$. 
\end{proof}

Equipped with the above lemmas, we now proceed with the proof of Theorem \ref{thm:k_general}. To this end, let $P = \{p_1, p_2, \ldots, p_n\}$ be a set of $n$ points in the plane in general position, where each point is assigned a color from the set $[c]:= \{1, 2, \ldots, c\}$. As in \eqref{eq:colorP}, let $P_a$ be the subset of points of color $a \in [c]$. Assume, without loss of generality $|P_1| \geq |P_2| \geq \cdots \geq |P_c|$. The proof proceeds recursively by removing points from $P$. To begin with, we set $P^{(0)} = P$, $\tilde n = \frac{n}{c(c+1)}$, and $K= \frac{1}{8 c^5}$. Clearly, $|P^{(0)}| \geq \tilde n$. Fix $t \geq 0$, and consider the set $P^{(t)}$ with $|P^{(t)}| \geq \tilde n \gtrsim n$. Note that if $\delta(P^{(t)})  >  K n^{\frac{1}{3}}$, then by Lemma \ref{lemma3}, 
$$X^{(c)}_{\leq (c-2)}(P^{(t)}) \gtrsim \frac{ \delta(P^{(t)}) |P^{(t)}|}{c} \gtrsim  n^{\frac{4}{3}} , $$
 as required. Hence, assume $\delta(P^{(t)}) \leq K n^{\frac{1}{3}}$. For $a \in [c]$, denote by $P^{(t)}_a \subseteq P^{(t)}$ the subset of points with color $a$. Then by Lemma \ref{lemma1},
\begin{align}\label{eq:P1}
\frac{|P^{(t)}|}{c} -  (c-1) K n^{\frac{1}{3}} \leq |P^{(t)}_a| \leq  \frac{|P^{(t)}|}{c} +  K n^{\frac{1}{3}} , 
\end{align}
for all $a \in [c]$. Throughout the recursion we will maintain the invariant $(P\setminus P^{(t)}) \cap CH(P^{(t)})= \emptyset$. Consequently, every triangle spanned by $P^{(t)}$ contains exactly the same points in its interior when regarded as a triangle
in $P^{(t)}$ or in the original point set $P$.  Now, build the convex hull of the points of color 1, remove all the points outside this convex hull and let $Q$ denote the remaining set of points, that is, 
$$Q= P^{(t)} \cap CH(P^{(t)}_1).$$  
%Since no point of color $1$ is removed in the construction of $Q$, we have $P_1^{(t)}\subseteq Q$. Moreover, by construction, $Q\subseteq CH(P_1^{(t)})$. Therefore, $P_1^{(t)} \subseteq Q \subseteq CH(P_1^{(t)})$, and hence, $CH(Q) = CH(P_1^{(t)})$. In particular, every vertex of $CH(Q)$ has color $1$. 
We now prove various properties of the set $Q$.

\begin{lemma}\label{lm:Q} 
Let $Q= Q_1 \cup Q_2 \cup \dots Q_c$ be the partition of $Q$ into the $c$ color classes. Then, for $K=K(c)$ as defined above, the following hold: 

\begin{itemize} 

\item[$(a)$] If $\delta(Q)   >  K n^{\frac{1}{3}}$, then $X^{(c)}_{\leq (c-2)}(P) \gtrsim n^{\frac{4}{3}}$. 
 
\item[$(b)$] If $\delta(Q)  \leq K n^{\frac{1}{3}}$, then 
\begin{align}\label{eq:Qc2}
\frac{|P^{(t)}|}{c} -  (c-1) K     n^{\frac{1}{3}} \leq |Q_1| \leq \frac{|P^{(t)}|}{c} + K n^{\frac{1}{3}} , 
\end{align}
and, for $a \in \{2, 3, \ldots, c\}$, 
\begin{align}\label{eq:Qa2c}
\frac{|P^{(t)}|}{c} -  ( 2 c-1) K     n^{\frac{1}{3}} \leq |Q_a| \leq \frac{|P^{(t)}|}{c} + K n^{\frac{1}{3}} . 
\end{align}

\end{itemize}

\end{lemma}

\begin{proof} Note that in the construction of $Q$, no point of color $1$ is removed from $P^{(t)}$, hence, $Q_1= P^{(t)}_1$. Therefore,  from \eqref{eq:P1}, 
\begin{align}\label{eq:Qc}
\lvert Q \rvert \geq \lvert Q_1 \rvert = \lvert P^{(t)}_1 \rvert \geq \frac{|P^{(t)}|}{c} -  (c-1) K     n^{\frac{1}{3}} \gtrsim n . 
\end{align}
Hence, if $\delta(Q)  > K n^{\frac{1}{3}}$, then applying Lemma $\ref{lemma3}$ gives, 
$$X^{(c)}_{\leq (c-2)}(P) \geq X^{(c)}_{\leq (c-2)}(Q) \gtrsim \frac{ \delta(Q) |Q|}{c}  \gtrsim n^{\frac{4}{3}} .$$
This proves Lemma \ref{lm:Q} (a). 

Now, assume that $\delta(Q)  \leq K n^{\frac{1}{3}}$. Then by Lemma \ref{lemma1}, for all $a \in [c]$, 
\begin{align}\label{eq:Qa}
|Q_a| \leq \frac{|Q|}{c} + K n^{\frac{1}{3}} \leq \frac{|P^{(t)}|}{c} + K n^{\frac{1}{3}} . 
\end{align}
Combining \eqref{eq:Qc} and \eqref{eq:Qa} shows \eqref{eq:Qc2}. 

Next, suppose $a \in \{2, 3, \ldots, c\}$. Then, 
\begin{align}
|Q_a| & \geq \frac{|Q|}{c} - (c-1) \delta(Q)  \tag*{(by the lower bound in Lemma \ref{lemma1})} \nonumber \\ 
& \geq |Q_1| - c \delta(Q) \tag*{(using $\frac{|Q|}{c} \geq |Q_1| -\delta(Q)$, from the upper bound in Lemma \ref{lemma1}) } \nonumber \\ 
& \geq |Q_1| - c K     n^{\frac{1}{3}} \tag*{(since $\delta(Q) \leq K n^{\frac{1}{3}}$)} \nonumber \\ 
& \geq \frac{|P^{(t)}|}{c} -  (2c-1) K     n^{\frac{1}{3}} , 
\label{eq:QaQ1}
\end{align}
where the last step uses the lower bound in \eqref{eq:Qc2}. Combining \eqref{eq:QaQ1} with the upper bound from \eqref{eq:Qa} shows \eqref{eq:Qa2c}. 
\end{proof}

Next, we show that if $\lvert V(CH(Q)) \rvert $ is large, then $Q \setminus V(CH(Q)$ has large discrepancy, hence, it contains the desired number of monochromatic triangles with at most $c-2$ interior points.

\begin{lemma} Suppose $\delta(Q)  \leq K n^{\frac{1}{3}}$. If $\lvert V(CH(Q)) \rvert > K' n^{\frac{1}{3}}$, where $K'=c(2c+1)K$, then, $\delta(Q \setminus V(CH(Q)) ) \geq K n^{\frac{1}{3}}$. Consequently, $X^{(c)}_{\leq (c-2)}(P) \gtrsim n^{\frac{4}{3}}$.  
\label{lm:Qc}
\end{lemma} 

\begin{proof} Let $Q' = Q \setminus V(CH(Q))$. Denote by $Q'_1, Q'_2, \ldots, Q'_c$ the $c$ color classes of $Q'$. Note that all points of $V(CH(Q))$ have color 1. Hence, $|Q'_1| \leq |Q_1| $ and $|Q'_a| = |Q_a| $, for $a \in \{2, 3, \ldots, c\}$. 
Then, for $a \in \{2, 3, \ldots, c\}$, the upper bound in \eqref{eq:Qa2c} gives, 
$$|Q'_a| = |Q_a| \leq \frac{|P^{(t)}|}{c} + K n^{\frac{1}{3}}.$$  Also,  the upper bound  in \eqref{eq:Qc2} and the assumption $\lvert V(CH(Q)) \rvert > K' n^{\frac{1}{3}}$ gives, 
$$|Q_1'| = |Q_1| - |V(CH(Q_1))| < \frac{|P^{(t)}|}{c} + K n^{\frac{1}{3}} - K'  n^{\frac{1}{3}} . $$
Hence, combining the above two inequalities, 
\begin{align}\label{eq:sum}
|Q'| = \sum_{a=1}^c |Q'_a| < |P^{(t)}| + c K n^{\frac{1}{3}} - K'  n^{\frac{1}{3}} . 
\end{align}
On the other hand, the lower bound in \eqref{eq:Qa2c} gives, 
\begin{align}\label{eq:maximum}
\max_{a \in [c]} |Q'_a| \geq |Q'_2|  \geq \frac{|P^{(t)}|}{c} -  (2c-1) K     n^{\frac{1}{3}} . 
\end{align}
Combining \eqref{eq:sum} and \eqref{eq:maximum}, 
$$\delta(Q') = \max_{a \in [c]} |Q'_a| -  \frac{|Q'|}{c} > - 2 c K     n^{\frac{1}{3}} + \frac{K'}{c} n^{\frac{1}{3}} = K n^{\frac{1}{3}} . $$
Moreover, from the lower bound in \eqref{eq:Qa2c}, 
\begin{align*}
|Q'| \geq \sum_{a=2}^c |Q'_a|  \geq \frac{c-1}{c} |P^{(t)}| -  (2c-1) (c-1) K     n^{\frac{1}{3}} \gtrsim n ,  
\end{align*}  
since $|P^{(t)}| \gtrsim n$. Then applying Lemma \ref{lemma3} gives, $X^{(c)}_{\leq (c-2)}(P) \geq X^{(c)}_{\leq (c-2)}(Q') \gtrsim \frac{ \delta(Q') |Q'|}{c} \gtrsim  n^{\frac{4}{3}}$.  
\end{proof}

Now, suppose $V(CH(Q)) = \{p_1, p_2, \ldots, p_R \}$ with points arranged in  clockwise order (as in Figure \ref{figure:vertex9}). Recall that all the points in $V(CH(Q))$ have color 1 and by the above lemmas it suffices to assume that $\delta(Q) \leq K n^{\frac{1}{3}}$ and $R \leq K' n^{\frac{1}{3}}$.  
Triangulate the convex hull $CH(Q)$ by adding the diagonals $p_1 p_{b+1}$, for $b \in \{ 2,\dots , R - 2\}$. Let $$T^{(b)}= \mathrm{int} (\Delta(p_1p_{b+1} p_{b+2}) ) , $$ for $b \in \{1,\dots, R - 2\}$.  Now, there are two cases:

\begin{figure}[ht]
     \centering
        \includegraphics[width=0.6\textwidth]{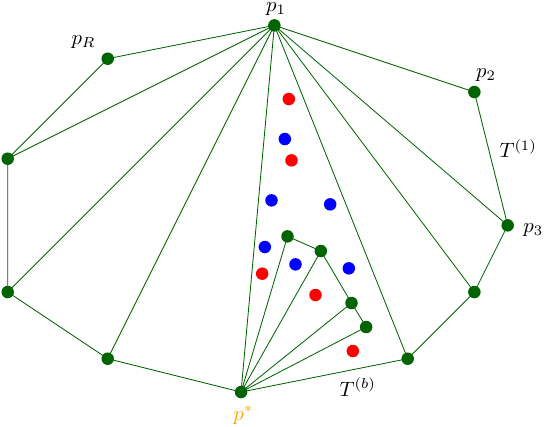}
         \caption{\small{Illustration for the proof of Theorem \ref{thm:k_general}. } }
         \label{figure:vertex9}
\end{figure}

\begin{itemize}

\item Suppose $\delta(T^{(b)} \cap Q )  >  K' n^{\frac{1}{3}}$, for some $b \in \{1, 2, \ldots, R-2\}$. 
Note that one of the three regions $T^{(b)}$, $T^{(1)}\ \cup \  T^{(2)} \dots \cup \    T^{(b-1)}$, and $T^{(b+1)}\  \cup \dots \cup \  T^{(R-2)}$ contains at least 
\begin{align} 
\frac {\sum_{a=2}^{c}|Q_a|+ (|Q_1|-R)}{3} & \geq  \frac{|P^{(t)}| -  2 c (c-1) K     n^{\frac{1}{3}} -R}{3} \tag*{(by \eqref{eq:Qc2} and \eqref{eq:Qa2c})}\nonumber \\ 
& \geq \frac{|P^{(t)}| -  2 c (c-1) K     n^{\frac{1}{3}} - c (2c+1)  K     n^{\frac{1}{3}} }{3} \tag*{(since $R \leq K' n^{\frac{1}{3}})$} \nonumber \\ 
& = \frac{|P^{(t)}|}{3}-\frac{c (4c - 1 ) K     n^{\frac{1}{3}} }{3} \gtrsim n , 
\label{eq:QPt} 
\end{align} 
points of $Q$. If $T^{(b)}$ is the region for which \eqref{eq:QPt} holds, then by Lemma \ref{lemma3}, $$X^{(c)}_{\leq (c-2)}(P) \geq X^{(c)}_{\leq (c-2)}(T^{(b)} \cap Q) \gtrsim \frac{ \delta(T^{(b)} \cap Q ) |T^{(b)} \cap Q |}{c} \gtrsim n^{\frac{4}{3}} .$$
%for the points inside $T_b$ and we get $\Omega (m^{\frac{4}{3}})$ monochromatic triangles with at most $k-2$ interior points.
Next, suppose $T^{(1)}\cup T^{(2)} \dots \cup T^{(b-1)}$ be the region such that \eqref{eq:QPt} holds. Then
\begin{align}\label{eq:Tb}
\delta((T^{(1)}\cup T^{(2)} \dots \cup T^{(b-1)}) \cap Q )  >  \frac{K' n^{\frac{1}{3}} }{c} \text{ or } \delta((T^{(1)}\cup T^{(2)} \dots \cup T^{(b)}) \cap Q )  >  \frac{K' n^{\frac{1}{3}} }{c} . 
\end{align} 
The proof of \eqref{eq:Tb} is given later in Section \ref{sec:trianglediscpf}. It follows from \eqref{eq:Tb} and Lemma \ref{lemma3} that either $X^{(c)}_{\leq (c-2)}((T^{(1)}\cup T^{(2)} \dots \cup T^{(b-1)}) \cap Q )  \gtrsim n^{\frac{4}{3}}$ or $X^{(c)}_{\leq (c-2)}((T^{(1)}\cup T^{(2)} \dots \cup T^{(b)}) \cap Q )  \gtrsim n^{\frac{4}{3}}$.  The same argument holds if \eqref{eq:QPt} holds for $T^{(b+1)}\cup T^{(b+2)} \dots \cup T^{(R-2)}$.

\item Now, suppose $\delta(T^{(b)} \cap Q ) \leq K' n^{\frac{1}{3}}$, for all $b \in \{1, 2, \ldots, R-2\}$. 
Note that there exists $b \in \{1, 2, \ldots, R-2\} $ such that $T^{(b)}_1$, which denotes the subset of points of $Q$ of color 1 in $T^{(b)}$, satisfies: 
$$|T^{(b)}_1| \geq \frac{|Q_1|-|V(CH(Q))|}{R-2} \geq  \frac{|Q_1|}{R} - 1 \geq \frac{|P^{(t)}|}{c K' n^{\frac{1}{3}}} -  \frac{(c-1) K}{K'} - 1 \geq \frac{n^{\frac{2}{3}}}{ 8 c^3 (c+1)^2 K } , $$  
for $n$ sufficiently large. By Lemma \ref{lemma2}, we can triangulate  $T^{(b)}_1 \cup \{p_1, p_{b+1}, p_{b+2}\}$, such that at least $\lvert T^{(b)}_1 \rvert + \sqrt{\lvert T^{(b)}_1 \rvert}$ triangles are adjacent to one of the vertices of the triangle $\Delta(p_1p_{b+1}p_{b+2})$. The vertices of these triangles have color 1, and at most 
$$ \frac{\sum_{a=2}^{c}| T^{(b)}_a| }{c-1} $$
of them are $(c-1)$-blocked.  Note that, from Lemma 3 it follows that, for $a \in \{2, 3, \ldots, c\}$
$$|T^{(b)}_a| \leq \frac{|T^{(b)} \cap Q |}{c} + \delta(T^{(b)} \cap Q ) , $$ 
and $|T^{(b)}_1| \geq \frac{|T^{(b)} \cap Q |}{c} - (c-1) \delta(T^{(b)} \cap Q)$.  
Hence, the number of triangles of color 1 that have one of $p_1, p_{b+1}, p_{b+2}$ as a vertex and which have at most $c-2$ interior points is at least  
\begin{align*}
\lvert T^{(b)}_1 \rvert + \sqrt{\lvert T^{(b)}_1 \rvert} - \frac{\sum_{a=2}^{c}| T^{(b)}_a | }{c-1} 
& \geq \frac{n^{\frac{1}{3}}}{ \sqrt{8 c^3 (c+1)^2 K} } - c \delta(T^{(b)} \cap Q) \nonumber \\ 
& \geq \frac{n^{\frac{1}{3}}}{ \sqrt{8 c^3 (c+1)^2 K} } - c K' n^{\frac{1}{3}} \nonumber \\ 
& = \frac{n^{\frac{1}{3}}}{ \sqrt{8 c^3 (c+1)^2 K} } -c^2 (2c+1) K n^{\frac{1}{3}} \nonumber \\ 
& = \left( \frac{c}{c+1} - \frac{2c+1}{8c^3} \right) n^{\frac{1}{3}} \nonumber \\
& \geq \tfrac{1}{2}n^{\frac{1}{3}} , 
\end{align*}  
since $K= \frac{1}{8 c^5}$ and $c \geq 2$. Hence, there exists a vertex $p^* \in \{ p_1, p_{b+1}, p_{b+2} \}$ such that there are at least 
$\frac{1}{6} n^{\frac{1}{3}}$ triangles of color 1 with at most $c-2$ interior points, with $p^*$ as a vertex (see Figure \ref{figure:vertex9}). Then define $P^{(t+1)} = Q \backslash \{p^*\}$, if $t < \lfloor\frac{n}{c+1}\rfloor$. If $t \geq \lfloor\frac{n}{c+1}\rfloor$, then stop the algorithm. To verify that the invariant is preserved, note that by the induction hypothesis, every point of $P\setminus P^{(t)}$ lies
outside $CH(P^{(t)})$, and hence outside $CH(P^{(t+1)})$. Moreover, every point of
$P^{(t)}\setminus Q$ lies outside $CH(Q) =
CH(P_1^{(t)})$, and therefore outside $CH(P^{(t+1)})$. Finally,
$p^*$ is a vertex of $CH(Q)$, so $p^*\notin CH(Q\setminus\{p^*\}) = CH(P^{(t+1)})$. 
Thus, $(P\setminus P^{(t+1)}) \cap CH(P^{(t+1)})= \emptyset$, as required.

\end{itemize} 
Note that, for $t < \lfloor\frac{n}{c+1}\rfloor$,  which means $t \leq \lfloor\frac{n}{c+1}\rfloor - 1 \leq \frac{n}{c+1} - 1$, 
\begin{align*}
|P^{(t+1)}| \geq |P^{(t)}_1| - 1 \geq |P_1| - (t+1) \geq \frac{n}{c} - (t+1) \geq  \frac{n}{c(c+1)} = \tilde{n} .  
\end{align*}
Note that the algorithm described above either terminates at step $t$, for some $t <  \lfloor\frac{n}{c+1}\rfloor$, or at step $\lfloor\frac{n}{c+1}\rfloor$. In the first case, the algorithm stops because we applied the discrepancy lemma to find $\gtrsim n^{\frac{4}{3}}$ monochromatic triangles with at most $c-2$ interior points. In the second
case, algorithm finds at least $\lfloor\frac{n}{c+1}\rfloor -1$ points, each of which is incident on at least $\frac{1}{6} n^{\frac{1}{3}}$ triangles of color 1 and has at most $c-2$ interior points. These points together give $\gtrsim n^{\frac{4}{3}}$ monochromatic triangles with at most $c-2$ interior points. This completes the proof of Theorem \ref{thm:k_general}.  \hfill $\Box$

\subsubsection{Proof of (\ref{eq:Tb}) } 
\label{sec:trianglediscpf}

Recall that we have to show that if $\delta(T^{(b)} \cap Q )  >  K' n^{\frac{1}{3}}$, for some $b \in \{1, 2, \ldots, R-2\}$, then \eqref{eq:Tb} holds. For notational convenience denote $S^{(1)} = T^{(b)} \cap Q$, 
\begin{align*}%\label{eq:A23}
S^{(2)} = (T^{(1)}\cup T^{(2)} \dots \cup T^{(b-1)}) \cap Q \text{ and } S^{(3)} = (T^{(1)}\cup T^{(2)} \dots \cup T^{(b)}) \cap Q . 
\end{align*} 
Note that $S^{(1)} \cup S^{(2)} = S^{(3)}$. Hence, for $a \in [c]$, $|S^{(3)}_a| = |S^{(1)}_a| + |S^{(2)}_a|$, where $S^{(1)}_a$ is the set of points of color $a$ in $S^{(1)}$ (and similarly for $S^{(2)}$ and $S^{(3)}$). 
Now, if possible, suppose that 
\begin{align}\label{eq:S1S2}
\delta(S^{(2)} ) \leq \frac{K' n^{\frac{1}{3}} }{c} \text{ and } \delta(S^{(3)} ) \leq \frac{K' n^{\frac{1}{3}} }{c} . 
\end{align} 
For a set $S$ and a color $a\in[c]$, define the signed
discrepancy of color $a$ in $S$ by $\partial_a(S):=|S_a|-\frac{|S|}{c}$. Then
$\sum_{a=1}^{c} \partial_a(S)=0$ and $\partial_a(S)\leq \delta(S)$, for every $a\in[c]$. Hence, 
\begin{align}
\label{eq:deltaSa}
-\partial_a(S) =\sum_{b \in [c] \backslash \{a \} } \partial_b(S)
\leq (c-1)\delta(S)  .  
\end{align}
Recalling that $S^{(3)}=S^{(1)} \cup S^{(2)}$, gives, for every $a\in[c]$, $\partial_a(S^{(1)}) = \partial_a(S^{(3)}) - \partial_a(S^{(2)})$. 
Using \eqref{eq:deltaSa}, we obtain
$$ \partial_a(S^{(1)}) \leq \delta(S^{(3)}) + (c-1)\delta(S^{(2)}).$$
Taking the maximum over $a\in[c]$ and applying \eqref{eq:S1S2} gives,  
$$\delta(S^{(1)}) \leq \delta(S^{(3)}) + (c-1)\delta(S^{(2)})  \leq
\frac{K'}{c}n^{\frac{1}{3}} + (c-1)\frac{K'}{c}n^{\frac{1}{3}} = K'n^{\frac{1}{3}}  ,  $$  
which is a contradiction. \hfill $\Box$

\subsection{Proof of Theorem \ref{thm:U} }
\label{sec:Upf}

Recall the definition of $Z_{=s}$ from Section \ref{sec:random}. Observe that 
\begin{align*} 
Z_{=s}(\{U_1, U_2, \ldots, U_n\}) = \sum_{1 \leq i < j < k \leq n} \bm 1
\{ |\mathrm{int}(\Delta(U_i, U_j, U_k)) \cap \{U_1, U_2, \ldots, U_n\}|= s \} .  
\end{align*} 
Then, by linearity of expectation, 
\begin{align}\label{eq:Un}
\mathbb E[Z_{=s}(\{U_1, U_2, \ldots, U_n\})] & = \sum_{1 \leq i < j < k \leq n} \mathbb P ( |\mathrm{int}(\Delta(U_i, U_j, U_k)) \cap \{U_1, U_2, \ldots, U_n\}|= s ) \nonumber \\ 
& = {n \choose 3} \mathbb P ( |\mathrm{int}(\Delta(U_1, U_2, U_3)) \cap \{U_1, U_2, \ldots, U_n\}|= s ) ,  
\end{align} 
by exchangeability of the random variables $U_1, U_2, \ldots, U_n$. Now, note that 
\begin{align}\label{eq:Uns}
& \mathbb P( |\mathrm{int}(\Delta(U_1, U_2, U_3)) \cap \{U_1, U_2, \ldots, U_n\}|= s ) \nonumber \\ 
& = {n -3 \choose s} \int_{[0, 1]^6}  \lambda_2(\Delta(u_1, u_2, u_3))^s (1- \lambda_2(\Delta(u_1, u_2, u_3)))^{n-s-3 } \mathrm du_1 \mathrm d u_2 \mathrm du_3 \\ 
& \sim  \frac{n^s}{ 2^s s! } \int_{[0, 1]^4} \left\{ \int_{[0, 1]^2} \left( \| u_1 - u_2 \| \cdot |h^{u_3}_{\widebar{u_1u_2}}| \right)^s  \left(1- \frac{1}{2}\| u_1 - u_2 \| \cdot |h^{u_3}_{\widebar{u_1u_2}}| \right)^{n-s-3} \mathrm d u_3 \right\} \mathrm du_1 \mathrm d u_2 , \nonumber 
\end{align}
where $\widebar{u_1 u_2}$ is the line joining $u_1$ and $u_2$, $\| u_1 - u_2\|$ is the Euclidean distance between $u_1$ and $u_2$, and $h^{u_3}_{\widebar{u_1u_2}}$ is the (signed) perpendicular distance from the point $u_3$ to  $\widebar{u_1 u_2}$. For fixed $u_1, u_2 \in [0, 1]^2$, define 
\begin{align}\label{eq:Gn}
G_n(u_1, u_2) = n^{s+1} \int_{[0, 1]^2} \left( \| u_1 - u_2 \| \cdot |h^{u_3}_{\widebar{u_1u_2}} | \right)^s  \left(1- \frac{1}{2}\| u_1 - u_2 \| \cdot | h^{u_3}_{\widebar{u_1u_2}} | \right)^{n-s-3} \mathrm d u_3. 
\end{align} 
Combining \eqref{eq:Un}, \eqref{eq:Uns}, and \eqref{eq:Gn} gives, 
\begin{align}\label{eq:GU}
\frac{\mathbb E[Z_{=s}(\{U_1, U_2, \ldots, U_n\})]}{n^2} \sim \frac{1}{ 6 } \cdot \frac{1}{ 2^s s! } \int_{[0, 1]^4} G_n(u_1, u_2) \mathrm d u_1 \mathrm d u_2 . 
\end{align} 
Hence, to prove \eqref{eq:XA2} it suffices to calculate the limiting value of $\int_{[0, 1]^4} G_n(u_1, u_2) \mathrm d u_1 \mathrm d u_2$. This is derived in the following lemma: 

%The result in \eqref{eq:XA2} now follows from the following lemma. 

\begin{lemma}\label{lm:Gn} For $G_n$ as defined in \eqref{eq:Gn}, 
\begin{align}\label{eq:Gn2}
\lim_{n \rightarrow \infty} \frac{1}{ 6 } \cdot \frac{1}{ 2^s s! }  \int_{[0, 1]^4} G_n(u_1, u_2) \mathrm d u_1 \mathrm d u_2 = 2  . 
\end{align}
\end{lemma}

\begin{proof} 
For fixed $h \in \mathbb R $, let $L_{\widebar{u_1u_2}}(h)$ be the line parallel to $\widebar{u_1u_2}$ with distance $|h|$ from $\widebar{u_1u_2}$. 
Then, by Fubini's theorem, \eqref{eq:Gn} can be rewritten as: 
\begin{align*}  
& \int_{[0, 1]^4} G_n(u_1, u_2) \mathrm d u_1 \mathrm du_2 \nonumber \\ 
& = n^{s+1}  \int_{[0, 1]^2} \int_{\mathbb R} \left( \| u_1 - u_2 \| |h| \right)^s \left(1- \frac{1}{2}\| u_1 - u_2 \| |h|  \right)^{n-s-3} \lambda_1(L_{\widebar{u_1u_2}}(h) \cap  [0, 1]^2 ) \mathrm d h  \mathrm d u_1 \mathrm du_2 .
\end{align*}
Consider the change of variable $h= \frac{z}{n}$. Then 
\begin{align}\label{eq:h}
& \int_{[0, 1]^4} G_n(u_1, u_2) \mathrm d u_1 \mathrm du_2 \nonumber \\ 
 & =  \int_{[0, 1]^4} \int_{\mathbb R} \left( \| u_1 - u_2 \| |z| \right)^s \left(1- \frac{1}{2}\| u_1 - u_2 \| \frac{|z|}{n}  \right)^{n-s-3} \lambda_1(L_{\widebar{u_1u_2}}\left(\frac{z}{n}\right) \cap [0, 1]^2) \mathrm d z .  
\end{align} 
Note that for every fixed $s \geq 0$, and almost every $u_1, u_2 \in [0, 1]^2$, and $z \in \mathbb R$, 
\begin{align}\label{eq:exponent}
 \left(1- \frac{1}{2}\| u_1 - u_2 \| \frac{|z|}{n}  \right)^{n-s-3} \lambda_1(L_{\widebar{u_1u_2}}\left(\frac{z}{n}\right) \cap [0, 1]^2)  & \rightarrow   e^{- \frac{1}{2}\| u_1 - u_2 \| |z|} \lambda_1(L_{\widebar{u_1u_2}}(0) \cap [0, 1]^2) . 
\end{align} 
Hence, by \eqref{eq:h} and \eqref{eq:exponent}, and the dominated convergence theorem give, 
\begin{align*}%\label{eq:h}
& \lim_{n \rightarrow \infty} \int_{[0, 1]^4} G_n(u_1, u_2) \mathrm d u_1 \mathrm du_2 \nonumber \\ 
 & =  \int_{[0, 1]^4} \| u_1 - u_2 \|^s  \lambda_1(L_{\widebar{u_1u_2}} (0) \cap [0, 1]^2) \left( \int_{\mathbb R} |z|^s e^{- \frac{1}{2}\| u_1 - u_2 \| |z|} \mathrm d z \right) \mathrm d u_1 \mathrm du_2 .  
\end{align*} 
Now, a standard Gamma integral computation shows, 
$$\int_{\mathbb R} |z|^s e^{- \frac{1}{2}\| u_1 - u_2 \| |z|} \mathrm d z = \frac{2^{s+2} s! }{\| u_1 - u_2 \|^{s+1}} . $$ 
Hence, 
$$ \lim_{n \rightarrow \infty} \frac{1}{6} \cdot \frac{1}{ 2^s s! } \int_{[0, 1]^4} G_n(u_1, u_2) \mathrm d u_1 \mathrm du_2 =  \frac{2}{3} \int_{[0, 1]^4} \frac{\lambda_1(L_{\widebar{u_1u_2}}(0) \cap [0, 1]^2)}{\| u_1 - u_2 \|}  \mathrm d u_1 \mathrm du_2 =  2 ,$$
where the last step follows from the Blaschke-Petkantschin formula in 2-dimensions (see \cite[Lemma 3.3]{reitzner2024stars}) and the identity from \cite[(8.56)]{schneider2008stochastic}. This completes the proof of Lemma \ref{lm:Gn}. 
\end{proof}

Combining \eqref{eq:GU} and \eqref{eq:Gn2}, the result in \eqref{eq:XA2} follows. The result in \eqref{eq:XA2s} follows from \eqref{eq:XA} and \eqref{eq:XA2}. This completes the proof of Theorem \ref{thm:U}. \hfill $\Box$

\section{Conclusions }

In this paper, we initiate the study of counting monochromatic triangles with a few interior points. In addition to improving the bounds obtained in this paper, several other questions emerge. We list a few of them below:

\begin{itemize}

\item Recall that for $c \geq 4$, \citet{cravioto2019almost} improved the upper bound in \eqref{eq:trianglecolorempty} to $c-3$. In particular, for $c = 4$, this implies that any 4-colored point set contains a monochromatic triangle with at most one interior point. Specifically, they showed that $\mathsf{M}_3(4,1) \leq 145$ and, more generally, that $\mathsf{M}_3(c, c-3) = O(c^2)$. This readily implies a linear lower bound on $X_{\leq (c-3)}^{(c)}(n)$. A natural question is to investigate whether this can be improved to a superlinear lower bound.

\item Another direction is to improve the bounds on $\lambda_3(c)$ (recall \eqref{eq:trianglecolorempty} and the following discussion). A related question is to understand the behavior of  $\lambda_3(c)$ for large $c$. Towards this, we expect the true value for large $c$ to match the lower bound in \eqref{eq:trianglecolorempty} up to smaller order terms, that is,  $\lim_{c \rightarrow \infty} \frac{\lambda_3(c)}{c} = \frac{1}{2}$.

\end{itemize}

\small

\noindent \textbf{Acknowledgements}: BBB was supported by NSF CAREER grant DMS 2046393 and a Sloan Research Fellowship. \\

\noindent \textbf{Data Availability}: The manuscript has no associated data. 

\bibliographystyle{plainnat} 
\bibliography{ref.bib}

\end{document}